\documentclass[letterpaper,10pt,conference]{ieeeconf}

\IEEEoverridecommandlockouts
\usepackage{amsmath,amssymb,amsfonts}
\usepackage{mathtools}
\usepackage{graphicx}
\usepackage{booktabs}
\usepackage{multirow}
\usepackage{algorithm}
\usepackage{algpseudocode}
\usepackage{xcolor}
\usepackage{bm}

\newcounter{proposition}

\newenvironment{proposition}
{%
    \refstepcounter{proposition}%
    \par\noindent
    \textbf{Proposition \theproposition.}\ 
}
{%
    \par
}

\newcommand{\R}{\mathbb{R}}

\newcommand{\sat}{\operatorname{sat}}

\newcommand{\ubar}{\underline{u}}
\newcommand{\xbar}{\underline{x}}
\newcommand{\uupper}{\overline{u}}
\newcommand{\xupper}{\overline{x}}

\title{Box-iLQR: Constraint-Aware Feedback via Finite-Barrier Trajectory Optimization}

\author{
Abhijeet$^{*}$ and Suman Chakravorty$^{*}$%
\thanks{$^{*}$Department of Aerospace Engineering, Texas A\&M University, College Station, TX, USA.}
}

\begin{document}

\maketitle
\thispagestyle{empty}
\pagestyle{empty}

\begin{abstract}

Iterative Linear Quadratic Regulator (iLQR) and Differential Dynamic
Programming (DDP) are attractive trajectory-optimization methods because
their backward recursions provide both nominal trajectories and local
time-varying feedback laws. Hard state and input constraints, however, can
drive solutions toward constraint boundaries where little corrective
authority remains during closed-loop execution.

This paper presents \emph{Box-iLQR}, a log-barrier iLQR method for
box-constrained states and controls, with emphasis on the resulting feedback
policy. Barrier gradients bias nominal trajectories away from constraint
boundaries, while barrier curvature enters the Riccati recursion that
determines the feedback gains. We show that control-barrier curvature
smoothly attenuates feedback as an actuator approaches saturation, while
state-barrier curvature propagates backward through the value function to
shape feedback before future state constraints are encountered. This motivates retaining a finite barrier rather than always driving the
barrier parameter toward zero, thereby preserving state and actuator margin.
Comparisons with CL-DDP and ALTRO are presented on pendulum, cart-pole,
Acrobot, and car-parking problems. Monte Carlo experiments demonstrate
improved closed-loop behavior with finite barriers, while a six-input MuJoCo
fish example demonstrates data-driven Box-iLQR using dynamics Jacobians
estimated from simulator rollouts. Box-iLQR achieves convergence comparable to, and in some cases better than, the baseline methods. Moreover, the feedback policy obtained with a finite barrier ensures constraint satisfaction even under execution uncertainty.

\end{abstract}

\section{Introduction}
\label{sec:intro}

Trajectory optimization is a fundamental tool for motion planning and control
of nonlinear robotic systems. Among shooting-based methods, Differential
Dynamic Programming (DDP) and the iterative Linear Quadratic Regulator (iLQR)
are particularly attractive because they exploit the temporal structure of the
optimal-control problem and produce a time-varying local feedback law together
with the optimized nominal trajectory
\cite{mayne1965ddp,ILQR_todorov,MNGM1}. Modern robotics software and
algorithms such as Crocoddyl further demonstrate the practical utility of
Riccati-based optimal-control methods for complex nonlinear systems
\cite{mastalli20crocoddyl}.

Among Riccati-based trajectory-optimization methods, we focus on iLQR rather
than full DDP. Recent Sequential Quadratic Programming (SQP) interpretations show that iLQR provides a consistent
quadratic approximation of the discrete optimal-control problem, while the
additional dynamics-Hessian terms in DDP can produce indefinite local models
and poor cost predictions away from a solution
\cite{abhi_DDP_ILQR,abhi_DDP_ILQR_ACC}. This makes iLQR particularly suitable
for the barrier-continuation setting considered here, where updates of the
barrier parameters repeatedly change the local objective and may move the
optimization away from its current minimizer.

The feedback returned by iLQR is not merely a numerical byproduct. A related
line of work views nonlinear control and reinforcement learning as the search
for a nominal maneuver together with a local feedback controller, rather than
as the direct search for a high-dimensional global policy
\cite{rand2c,wang2022d2c,MNGM1}. In this view, trajectory
optimization determines the operating trajectory, while local perturbation
dynamics determine the feedback correction around it.

The central difficulty is that practical systems are constrained. Actuators
saturate, joint or workspace coordinates are bounded, and safety requirements
can impose explicit limits on states. CL-DDP incorporates box-constrained
controls through stagewise constrained quadratic subproblems
\cite{CL_DDP}. Feasibility-driven DDP methods improve globalization and
initialization behavior \cite{mastalli2020fddp}, while hybrid and projection
methods offer alternative constrained formulations
\cite{lantoine2012hybrid,Giftthaler2017APA,Xie2017DifferentialDP}. ALTRO
combines an augmented-Lagrangian formulation with structure-exploiting
trajectory optimization to address general nonlinear constraints
\cite{howell2019altro}. Interior-point methods provide another important class of constrained
optimization techniques, including structured optimal-control solvers and
interior-point variants of DDP
\cite{rao1998application,DDP_IPM,FATROP,pavlov2021ipddp}.
A key distinction of the primal log-barrier formulation used in Box-iLQR is
that every accepted iterate remains strictly inside the feasible region.
In contrast, augmented-Lagrangian methods may pass through infeasible
intermediate trajectories while progressively reducing the constraint
residual. Primal-dual interior-point methods likewise do not require
strict primal feasibility at every intermediate iteration
\cite{FATROP,DDP_IPM,pavlov2021ipddp}. Box-iLQR, therefore, preserves strict
feasibility throughout the accepted optimization sequence, at the cost of
requiring a strictly feasible initialization. Moreover, Box-iLQR offers a critical advantage for feedback-oriented deployment:
the log-barrier curvature directly shapes the Riccati gains, ensuring
closed-loop system satisfies the constraints under uncertainty. 

Pure primal barrier methods, however, introduce an important numerical issue.
Wright showed that after reducing a barrier parameter, a Newton step can become
strongly infeasible even when computed from an interior iterate
\cite{wright1995primalbarrier}. Classical nonlinear optimization therefore
combines barrier models with globalization mechanisms such as line search or
trust regions \cite{nocedal2006numerical}. Similar behavior is relevant to
trajectory optimization because a control perturbation at one time instant
affects all downstream states.

This paper takes a complementary viewpoint: if the output of trajectory
optimization will be used directly as a \emph{feedback controller}, driving
the barrier parameter arbitrarily close to zero may not be desirable.
At a hard-constrained optimum, a control can lie almost exactly at its
saturation limit, leaving little authority for disturbance rejection in one
direction. Likewise, a nominal state that approaches a path-constraint
boundary has little geometric margin against execution uncertainty. Retaining
a finite logarithmic barrier intentionally shifts the nominal solution into
the feasible interior and simultaneously modifies the local curvature from
which the feedback matrix is computed.


The principal contributions of this work are:
\begin{itemize}

    \item A finite-barrier Box-iLQR formulation that handles simultaneous state and control box constraints while preserving the standard local feedback policy parametrization of iLQR/DDP leading to the structured backward-forward search.

    \item A characterization of how barrier curvature shapes the feedback gains. As a constrained control approaches its bound, the corresponding row of the feedback matrix tends to zero. On the other hand, the state-barrier curvature propagates backward  and modifies the feedback before the state constraint is reached thereby allowing for constraint satisfaction even in the presence of uncertainty.

    \item A finite-barrier stopping rule for feedback deployment. Instead of
    driving the barrier parameter to zero, the optimization may
    terminate at a positive barrier value to retain state and actuator margins
    for closed-loop execution that ensures constraint satisfaction.

    \item Numerical validation against CL-DDP and ALTRO, together with a
data-driven MuJoCo example in which the local dynamics are identified from
simulator rollouts. Experiments show that Box-iLQR is competitive/ better than CL-DDP
and ALTRO on nominal constrained otpimal control, while the finite-barrier feedback
reduces observed state-constraint violations to zero in problems under uncertainty, which the competing methods cannot accomplish.

\end{itemize}

\section{Problem Formulation}
\label{sec:problem}

Consider the discrete-time finite-horizon optimal-control problem

\begin{subequations}
\label{eq:ocp}
\begin{align}
\min_{\{u_k\}_{k=0}^{N-1}}
\quad &
J =
\phi(x_N)
+
\sum_{k=0}^{N-1}
\ell(x_k,u_k),
\\
\text{s.t.}\quad
&
x_{k+1}=f(x_k,u_k),
\qquad
k=0,\ldots,N-1,
\\
&
\xbar_i < x_{k,i} < \xupper_i,
\qquad i\in\mathcal I_x,
\\
&
\ubar_j < u_{k,j} < \uupper_j,
\qquad j\in\mathcal I_u .
\end{align}
\end{subequations}

Here $x_k\in\R^n$, $u_k\in\R^m$, and $x_0$ is fixed. Only selected state or
control components need to be constrained. We assume that $f$ is continuously
differentiable and that the nominal running and terminal costs are twice
continuously differentiable over the region visited by the optimizer.

\subsection{Continuous- and discrete-time state constraints}

If the original plant is continuous time,

\begin{equation}
    \dot x=f_c(x,u),
\end{equation}

and $u(t)=u_k$ on
$I_k=[t_k,t_{k+1}]$, let $F_{\Delta t}$ denote the exact zero-order-hold flow
map. Then

\begin{equation}
    x_{k+1}=F_{\Delta t}(x_k,u_k)
\end{equation}

is an exact sampled representation of the continuous dynamics.

A zero order control hold means the control is not violated for continuous dynamics as well. Although for states, we only have endpoint feasibility as we convert continuous dynamics to discrete. Endpoint feasibility alone does not generally imply continuous-time path
feasibility. For a constrained state component $x_i$, assume additionally that
$x_i(t)$ is monotonic on each sampling interval, $I_k$. Under this explicit assumption,

\begin{equation}
\xbar_i
\leq x_{k,i},x_{k+1,i}
\leq \xupper_i
\quad
\Longleftrightarrow
\quad
\xbar_i\leq x_i(t)\leq\xupper_i,
\;\;
t\in I_k .
\label{eq:monotone_equiv}
\end{equation}

With numerical integration,
\eqref{eq:monotone_equiv} holds only up to the integration error. This
qualification is important whenever discrete nodes are used to represent
continuous-time state constraints
\cite{abhi_boxilqr_preprint}.

A continuous-time performance index in Lagrange form can be readily approximated by the discrete cost in \eqref{eq:ocp} using a numerical transcription scheme such as Euler or Runge--Kutta integration.

\subsection{Log-barrier subproblem}

For state-barrier parameters $\mu_i>0$ and control-barrier parameters
$\sigma_j>0$, define

\begin{align}
b_x(x_k;\mu)
&=
-\sum_{i\in\mathcal I_x}
\mu_i
\left[
\log(x_{k,i}-\xbar_i)
+
\log(\xupper_i-x_{k,i})
\right],
\label{eq:bx}
\\
b_u(u_k;\sigma)
&=
-\sum_{j\in\mathcal I_u}
\sigma_j
\left[
\log(u_{k,j}-\ubar_j)
+
\log(\uupper_j-u_{k,j})
\right].
\label{eq:bu}
\end{align}

The barrier-augmented running cost is

\begin{equation}
\ell_b(x_k,u_k)
=
\ell(x_k,u_k)
+
b_x(x_k;\mu)
+
b_u(u_k;\sigma).
\label{eq:augcost}
\end{equation}

The terminal state barrier is included whenever $x_N$ is constrained. Because the logarithm is defined only in the strict interior, every accepted
Box-iLQR iterate remains strictly feasible with respect to the discrete box
constraints. This differs from penalty and augmented-Lagrangian approaches,
which can admit infeasible intermediate iterates while driving the constraint
residual toward zero. Thus, feasibility is enforced directly by the domain of
the optimization problem rather than only through a penalty on constraint
violation.

For a constrained scalar control $u_j$,

\begin{align}
(b_u)_{u_j}
&=
-\sigma_j
\left[
\frac{1}{u_j-\ubar_j}
-
\frac{1}{\uupper_j-u_j}
\right],
\\
(b_u)_{u_ju_j}
&=
\sigma_j
\left[
\frac{1}{(u_j-\ubar_j)^2}
+
\frac{1}{(\uupper_j-u_j)^2}
\right]
>0 .
\label{eq:barrier_hessian}
\end{align}

Analogous expressions hold for constrained state components. The barrier
therefore contributes both a repulsive gradient and a rapidly increasing
positive curvature near each boundary, as in standard logarithmic
interior-point methods \cite{nocedal2006numerical}.

\section{Box-iLQR}
\label{sec:method}

Box-iLQR retains the standard backward--forward structure of iLQR while
incorporating the state and control barriers defined in
Sec.~\ref{sec:problem}. As summarized in Algorithm~\ref{alg:boxilqr}, each
iteration first performs a backward pass along the current nominal trajectory
to compute a local feedforward correction and feedback policy. A nonlinear
forward rollout is then performed using this policy, with line search used to
ensure sufficient decrease and strict feasibility. Previous work provides the convergence analysis and associated theoretical
guarantees for the underlying log-barrier formulation~\cite{abhi_boxilqr_preprint}.

Let $(\bar x_k,\bar u_k)$ denote the current nominal trajectory. Linearizing
the dynamics about this trajectory gives
\begin{equation}
    \delta x_{k+1}
    =
    A_k\delta x_k+B_k\delta u_k,
    \label{eq:linear_dynamics}
\end{equation}
where
$
    A_k=f_x(\bar x_k,\bar u_k),
    \quad \text{and}\quad
    B_k=f_u(\bar x_k,\bar u_k).
$

During the \emph{backward pass} of Algorithm~\ref{alg:boxilqr}, the
barrier-augmented stage cost is quadratically approximated along the nominal
trajectory. Combining this approximation with the value function at
$k+1$ gives
\begin{align}
    Q_x
    &=
    \ell_{b,x}+A_k^\top V_x^{k+1},
    &
    Q_u
    &=
    \ell_{b,u}+B_k^\top V_x^{k+1},
    \label{eq:QxQu}
    \\
    Q_{xx}
    &=
    \ell_{b,xx}+A_k^\top V_{xx}^{k+1}A_k,
    &
    Q_{uu}
    &=
    \ell_{b,uu}+B_k^\top V_{xx}^{k+1}B_k,
    \label{eq:QxxQuu}
    \\
    Q_{ux}
    &=
    \ell_{b,ux}+B_k^\top V_{xx}^{k+1}A_k.
    \label{eq:Qux}
\end{align}
The terminal conditions on $V_{x}$ and $V_{xx}$ are 
\begin{align*}
    V_{x}^{N} = \phi_{x},\quad \text{and} \quad  V_{xx}^{N} = \phi_{xx}.
\end{align*}
When required, the control Hessian is regularized as
\begin{equation}
    \widetilde Q_{uu}
    =
    Q_{uu}+\lambda I,
    \qquad
    \lambda\geq0,
    \label{eq:regularized_quu}
\end{equation}
so that it remains positive definite.

Minimization of the local quadratic model gives the affine control correction
\begin{equation}
    \delta u_k
    =
    k_k+K_k\delta x_k,
    \label{eq:local_policy}
\end{equation}
where
\begin{equation}
    k_k
    =
    -\widetilde Q_{uu}^{-1}Q_u,
    \qquad
    K_k
    =
    -\widetilde Q_{uu}^{-1}Q_{ux}.
    \label{eq:feedback}
\end{equation}
The value-function derivatives are then propagated backward in time using the
usual Riccati recursion. Because the barrier derivatives are included in
$\ell_b$, both state and control constraints directly modify the same
backward recursion that produces $k_k$ and $K_k$.

During the \emph{forward pass} of Algorithm~\ref{alg:boxilqr}, the nonlinear
system is rolled out using
\begin{align}
    u_k^+(\alpha)
    &=
    \bar u_k
    +
    \alpha k_k
    +
    K_k(x_k^+-\bar x_k),
    \label{eq:forward_u}
    \\
    x_{k+1}^+
    &=
    f(x_k^+,u_k^+),
    \qquad
    x_0^+=x_0.
    \label{eq:forward_x}
\end{align}
Here $\alpha\in(0,1]$ scales the feedforward update while retaining the
feedback correction. Before the barrier cost is evaluated, the complete
nonlinear rollout is checked for strict satisfaction of all box constraints.
If the rollout is infeasible or does not provide sufficient decrease,
Algorithm~\ref{alg:boxilqr} reduces the step according to
\begin{equation}
    \alpha
    \leftarrow
    \beta_\alpha\alpha,
    \qquad
    0<\beta_\alpha<1.
    \label{eq:alpha_backtrack}
\end{equation}
Only when feasibility cannot be recovered for
$\alpha\geq\bar\alpha$ is the corresponding barrier parameter increased and
the backward pass recomputed. Thus, the trajectory update is first
globalized through line search before the barrier model itself is modified.
This ordering is motivated by the infeasible-Newton behavior of pure primal
barrier methods \cite{wright1995primalbarrier}.


Full DDP additionally includes contractions between the value gradient and
second derivatives of the dynamics. These terms can improve local convergence near a solution, but away from the
optimum they need not preserve positive curvature and can produce a poor local
model. Recent work shows that iLQR corresponds to a consistent
SQP approximation of the discrete optimal-control problem and
retains stronger descent properties
\cite{abhi_DDP_ILQR,abhi_DDP_ILQR_ACC}. Since barrier continuation repeatedly
changes the local objective, we use the more robust iLQR approximation.

Algorithm~\ref{alg:boxilqr} summarizes the feasibility-first continuation
strategy. If no feasible rollout is recovered, the barrier associated with
the violated constraint is increased; if feasible rollouts exist but none
provides sufficient decrease, the regularization is increased instead.
After convergence of the fixed-barrier subproblem, the barrier parameters
are reduced toward $\gamma_{\mathrm{stop}}$. For nominal constrained
optimization, $\gamma_{\mathrm{stop}}$ is chosen near zero, whereas for
feedback deployment continuation may terminate at a finite
$\gamma_{\mathrm{fb}}>0$.
\begin{algorithm}[t]
\caption{Box-iLQR}
\label{alg:boxilqr}
\begin{algorithmic}[1]

\Require Strictly feasible $(\bar x,\bar u)$,
barriers $\gamma=[\mu^\top,\sigma^\top]^\top>0$

\While{termination condition not satisfied}

    \State Linearize dynamics and quadratize the
    barrier-augmented cost

    \State Backward pass: compute
    $(k_k,K_k)_{k=0}^{N-1}$ using
    \eqref{eq:QxQu}--\eqref{eq:feedback}

    \State $\alpha\gets1$, $\mathrm{accepted}\gets\mathrm{false}$

    \While{$\alpha\geq\bar\alpha$}

        \State Roll out
        \eqref{eq:forward_u}--\eqref{eq:forward_x}

        \If{rollout is infeasible}
            \State $\alpha\gets\beta_\alpha\alpha$

        \ElsIf{sufficient decrease is satisfied}
            \State $(\bar x,\bar u)\gets(x^+,u^+)$
            \State $\mathrm{accepted}\gets\mathrm{true}$
            \State \textbf{break}

        \Else
            \State $\alpha\gets\beta_\alpha\alpha$
        \EndIf

    \EndWhile

    \If{no feasible rollout was found}
        \State Increase offending barrier parameter(s)

    \ElsIf{\textbf{not} $\mathrm{accepted}$}
        \State Increase regularization $\lambda$

    \ElsIf{fixed-barrier problem has converged}
        \State Reduce $\gamma$ toward $\gamma_{\mathrm{stop}}$
    \EndIf

\EndWhile

\State \Return $(x^\star,u^\star,K^\star)$

\end{algorithmic}
\end{algorithm}

\section{Finite Barriers as Feedback Shaping}
\label{sec:feedback}

The logarithmic barrier does more than enforce feasibility during trajectory
optimization: its gradient modifies the nominal trajectory, while its
curvature directly affects the feedback gains computed in the backward pass.

Consider a scalar constrained quantity
$
    z\in(\underline z,\overline z)
$
with barrier
\begin{equation}
    b(z;\gamma)
    =
    -\gamma
    \left[
        \log(z-\underline z)
        +
        \log(\overline z-z)
    \right].
\end{equation}
Its descent contribution is
\begin{equation}
    -\nabla b
    =
    \gamma
    \left[
        \frac{1}{z-\underline z}
        -
        \frac{1}{\overline z-z}
    \right],
    \label{eq:push}
\end{equation}
which increasingly repels the nominal trajectory from a nearby boundary.
Along the central path, perturbed complementarity gives approximately
\begin{equation}
    s
    \approx
    \frac{\gamma}{\lambda^\star},
    \label{eq:slack}
\end{equation}
where $s$ is the active constraint slack and $\lambda^\star$ is the
corresponding Lagrange multiplier
\cite{nocedal2006numerical} for hard constraint. Hence, retaining $\gamma>0$ deliberately
preserves finite interior margin.

For a constrained input component $u_{k,j}$, write the control
Hessian as
\begin{equation}
    \widetilde Q_{uu,k}
    =
    H_k+D_{b,k},
    \label{eq:Quu_decomposition}
\end{equation}
where $H_k$ contains the non-barrier terms, and $D_{b,k}$
is the diagonal barrier Hessian. Its $j$th diagonal entry is
\begin{equation}
    [D_{b,k}]_{jj}
    =
    \sigma_j
    \left[
        \frac{1}{(u_{k,j}-\underline u_j)^2}
        +
        \frac{1}{(\overline u_j-u_{k,j})^2}
    \right].
    \label{eq:Dbjj}
\end{equation}

\begin{proposition}\label{prop:atten}
Suppose $H_k$ and $Q_{ux,k}$ remain bounded and the block of
$\widetilde Q_{uu,k}$ associated with the remaining controls remains uniformly
nonsingular. For fixed $\sigma_j>0$, if $u_{k,j}$ approaches either of its
bounds, then
\begin{equation}
    K_{k,j:}\rightarrow0.
    \label{eq:K_zero}
\end{equation}
\end{proposition}
\begin{proof}
Let
\begin{equation}
    d_j=[D_{b,k}]_{jj}.
\end{equation}
From \eqref{eq:Dbjj}, $d_j\rightarrow\infty$ as the corresponding slack tends
to zero. After permuting the $j$th control to the first coordinate,
\begin{equation}
    \widetilde Q_{uu,k}
    =
    \begin{bmatrix}
        d_j+\mathcal O(1) & b^\top\\
        b & C
    \end{bmatrix},
\end{equation}
where $b$ and $C^{-1}$ remain bounded. The block inverse formula therefore
gives
\begin{equation}
    e_j^\top
    \widetilde Q_{uu,k}^{-1}
    =
    \mathcal O(d_j^{-1}).
\end{equation}
Since $Q_{ux,k}$ is bounded,
\begin{equation}
    K_{k,j:}
    =
    -e_j^\top
    \widetilde Q_{uu,k}^{-1}Q_{ux,k}
    \rightarrow0.
\end{equation}
\end{proof}

For a slack $s_j$ from the approaching bound,
\begin{equation}
    [D_{b,k}]_{jj}
    \approx
    \left(
        \frac{\sigma_j}{s_j^2}
    \right),
\end{equation}
so Proposition~1 gives
\begin{equation}
    \|K_{k,j:}\|
    =
    \mathcal O
    \left(
        \frac{s_j^2}{\sigma_j}
    \right),
    \qquad
    s_j\rightarrow0.
    \label{eq:K_slack_rate}
\end{equation}

CL-DDP exhibits a related limiting behavior through its active-set
formulation. When a control is clamped at a bound, the corresponding feedback
row is set to zero \cite{CL_DDP}. Box-iLQR reaches the same limiting structure
continuously: as the nominal input approaches saturation, the logarithmic
barrier curvature increases and progressively attenuates the corresponding
feedback gain. Thus, CL-DDP changes the local feedback structure through a
discrete active-set decision, whereas Box-iLQR modifies it smoothly through
the barrier Hessian.

State barriers influence the feedback policy through the same backward pass.
Their curvature enters $\ell_{b,xx}$ and is propagated backward through the
value Hessian $V_{xx}$. At an earlier time step,
\begin{align}
    Q_{uu}
    &=
    \ell_{b,uu}
    +
    B^\top V_{xx}^{k+1}B,
    \\
    Q_{ux}
    &=
    \ell_{b,ux}
    +
    B^\top V_{xx}^{k+1}A,
\end{align}
and therefore
\begin{equation}
    K
    =
    -Q_{uu}^{-1}Q_{ux}.
\end{equation}
Consequently, curvature generated by a future state constraint influences
earlier feedback gains before the nominal trajectory reaches the constrained
region. We refer to this effect as constraint-aware or anticipatory feedback.

Finally, we distinguish between a barrier used to approximate the
hard-constrained optimum and one retained for feedback deployment:
\begin{align}
    (x_{\mathrm{opt}},u_{\mathrm{opt}},K_{\mathrm{opt}})
    &:\
    \gamma_{\mathrm{stop}}
    =
    \epsilon_{\mathrm{opt}},
    \qquad
    \epsilon_{\mathrm{opt}}\ll1,
    \\
    (x_{\mathrm{fb}},u_{\mathrm{fb}},K_{\mathrm{fb}})
    &:\
    \gamma_{\mathrm{stop}}
    =
    \gamma_{\mathrm{fb}}>0.
\end{align}
The first provides a closer approximation to the hard-constrained nominal
solution. The second intentionally retains state and actuator margin and is
used when the feedback policy itself is an important output of the
optimization. The finite barrier therefore trades nominal optimality for
interior margin and potentially improved closed-loop robustness.


Although we do not provide guarantees of constraint satisfaction under uncertainty, we evaluate its closed-loop effect empirically in
Sect.~\ref{sec:results}.

\section{Experiments}
\label{sec:results}

Box-iLQR is evaluated on pendulum, cart-pole, Acrobot, and car-parking
benchmarks and on a data-driven MuJoCo fish model. For the baseline comparisons, ALTRO is implemented using the open-source
\texttt{Altro.jl} package~\cite{howell2019altro}, while CL-DDP is implemented
following the control-limited DDP formulation of Tassa \emph{et al.}~\cite{CL_DDP}. Reported iteration counts are solver-specific convergence
statistics and should not be interpreted as direct computational-time
comparisons.

The pendulum, cart-pole, and Acrobot controls satisfy $|u|\leq1$,
$|u|\leq2$, and $|u|\leq5$, respectively. For the state-constrained
cart-pole, the cart position additionally satisfies $|x|\leq0.15$~m. The
car has componentwise control bounds $|u_1|\leq0.5$ and $|u_2|\leq2$; its
state-constrained variant additionally requires
$|x|,|y|\leq1.5$~m. The control-only car-parking task uses
$x_g=[0,0,0,0]^\top$ as in \cite{CL_DDP}, whereas the state-constrained task uses
$x_g=[0,0,2\pi,0]^\top$.

\subsection{Benchmark Comparison}

\begin{table*}[t]
\centering
\caption{BENCHMARK COMPARISON. COSTS ARE EVALUATED USING THE COMMON NOMINAL OBJECTIVE AND EXCLUDE BARRIER, PENALTY, AND AUGMENTED-LAGRANGIAN TERMS. ( N/A - CL-DDP cannot accomodate state constraints. )}
\label{tab:solver_comparison}

\resizebox{\textwidth}{!}{
\begin{tabular}{llrrrrrr}
\toprule
&
&
\multicolumn{2}{c}{ALTRO}
&
\multicolumn{2}{c}{CL-DDP}
&
\multicolumn{2}{c}{Box-iLQR}
\\
\cmidrule(lr){3-4}
\cmidrule(lr){5-6}
\cmidrule(lr){7-8}

Problem
&
Constraints
&
Iterations & Cost
&
Iterations & Cost
&
Iterations & Cost
\\
\midrule

Pendulum
&
Control
&
436 & 67.24
&
206 & \textbf{55.46}
&
527 & \textbf{55.49}
\\

Cart-pole
&
Control
&
1091 & 994.92
&
3554 & \textbf{762.44}
&
8226 & \textbf{762.45}
\\

Cart-pole
&
State + control
&
968 & 912.38
&
N/A & --
&
3729 & \textbf{765.05}
\\

Acrobot
&
Control
&
4665 & \textbf{39581.50}
&
4797 & 41712.44
&
4096 & 40494.38
\\

Car parking
&
Control
&
142 & \textbf{1.4361}
&
1597 & 1.9448
&
335 & 1.9065
\\

Car parking
&
State + control
&
113 & 0.70
&
N/A & --
&
954 & \textbf{0.69}
\\

\bottomrule
\end{tabular}}
\vspace{-0.5cm}
\end{table*}

Table~\ref{tab:solver_comparison} shows that Box-iLQR obtains solutions
comparable to the established constrained solvers on the control-limited
benchmarks. For the pendulum, Box-iLQR and CL-DDP obtain nearly identical
costs, $55.49$ and $55.46$, respectively. The same behavior is observed for
the control-constrained cart-pole, where the corresponding costs are
$762.45$ and $762.44$. For control-only car parking, ALTRO obtains the lowest reported objective,
1.4361, followed by Box-iLQR at 1.9065 and CL-DDP at 1.9448. ALTRO converged to a better solution than Box-iLQR/CL-DDP.

All three methods converge for the Acrobot. ALTRO obtains the lowest cost,
$39581.50$, followed by Box-iLQR at $40494.38$ and CL-DDP at $41712.44$.
Figure~\ref{fig:acrobot_comparison} compares the resulting trajectories. The
different objective values emphasize that the methods are local nonlinear
optimizers and need not converge to identical stationary solutions even from
the same task initialization.

\begin{figure}[t]
    \centering
    \includegraphics[width=\columnwidth]
    {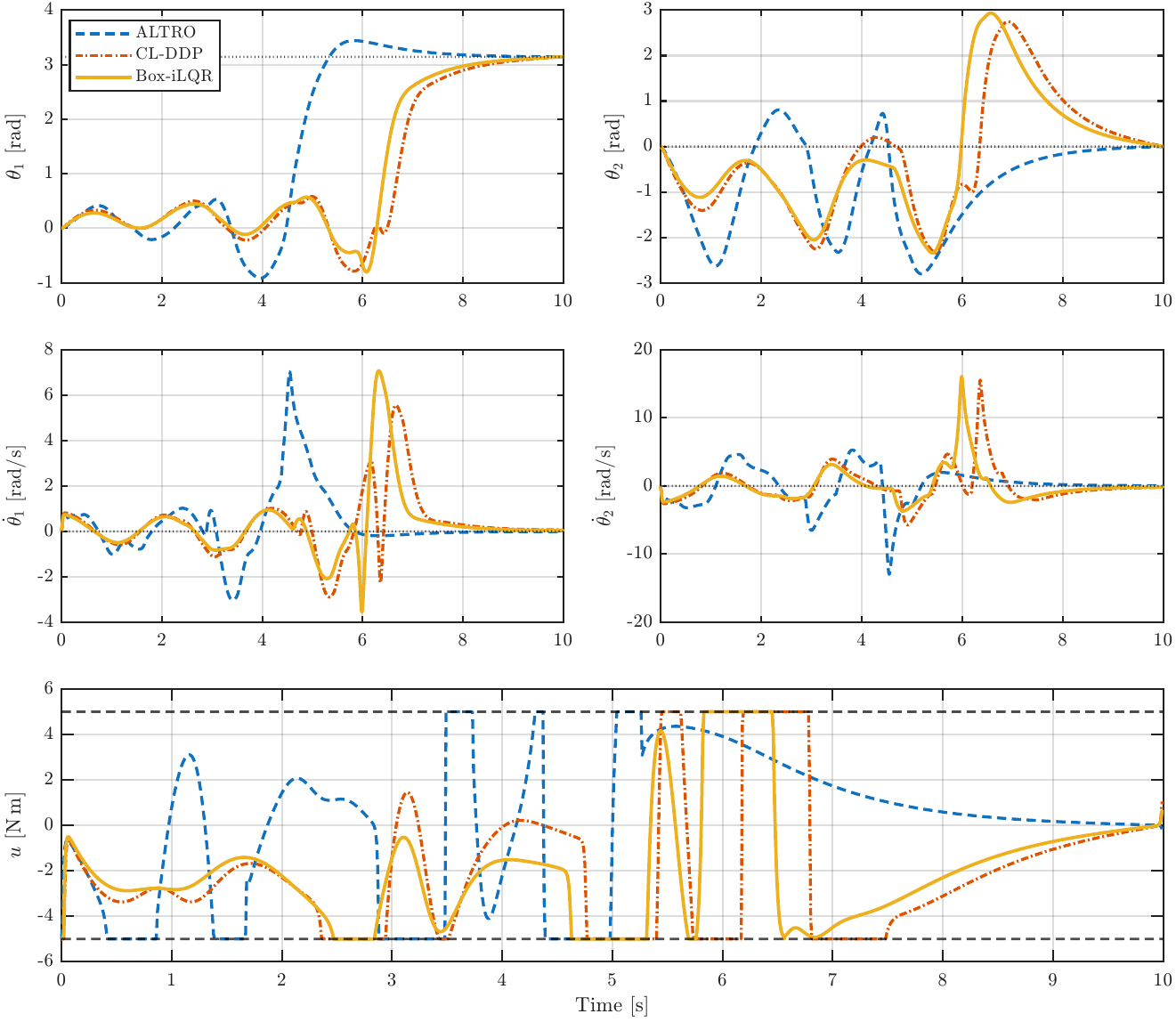}
    \caption{Acrobot swing-up using ALTRO, CL-DDP, and Box-iLQR. All three
    solutions satisfy the actuator constraint $|u|\leq5$.}
    \label{fig:acrobot_comparison}
\end{figure}

The main distinction appears when state constraints are introduced. For the
state-constrained cart-pole, ALTRO converges to a feasible solution
with higher cost, $912.38$, while the CL-DDP formulation considered here does not
accommodate state constraints. Box-iLQR converges with cost $765.05$.
For state-constrained car parking, both ALTRO and Box-iLQR converge to
constraint-feasible trajectories, while the CL-DDP formulation considered
here is not applicable. As shown in Fig.~\ref{fig:car_state_constrained},
the ALTRO and Box-iLQR reach the desired terminal
configuration $x_g=[0,0,2\pi,0]^\top$, with a final cost of $0.70$ and $0.69$, respectively. For the state-constrained car-parking task, ALTRO was sensitive to
initialization and solver parameters in our experiments; the solution
reported in Table~I is the converged constraint-feasible result used for
comparison.

The
closed-loop effect of the resulting feedback law is examined later in
Fig.~\ref{fig:car_feedback}, where Box-iLQR feedback is shown to eliminate
the state-constraint violations observed under open-loop execution in the
tested Monte Carlo set.

\begin{figure}[t]
    \centering
    \includegraphics[width=\columnwidth]
    {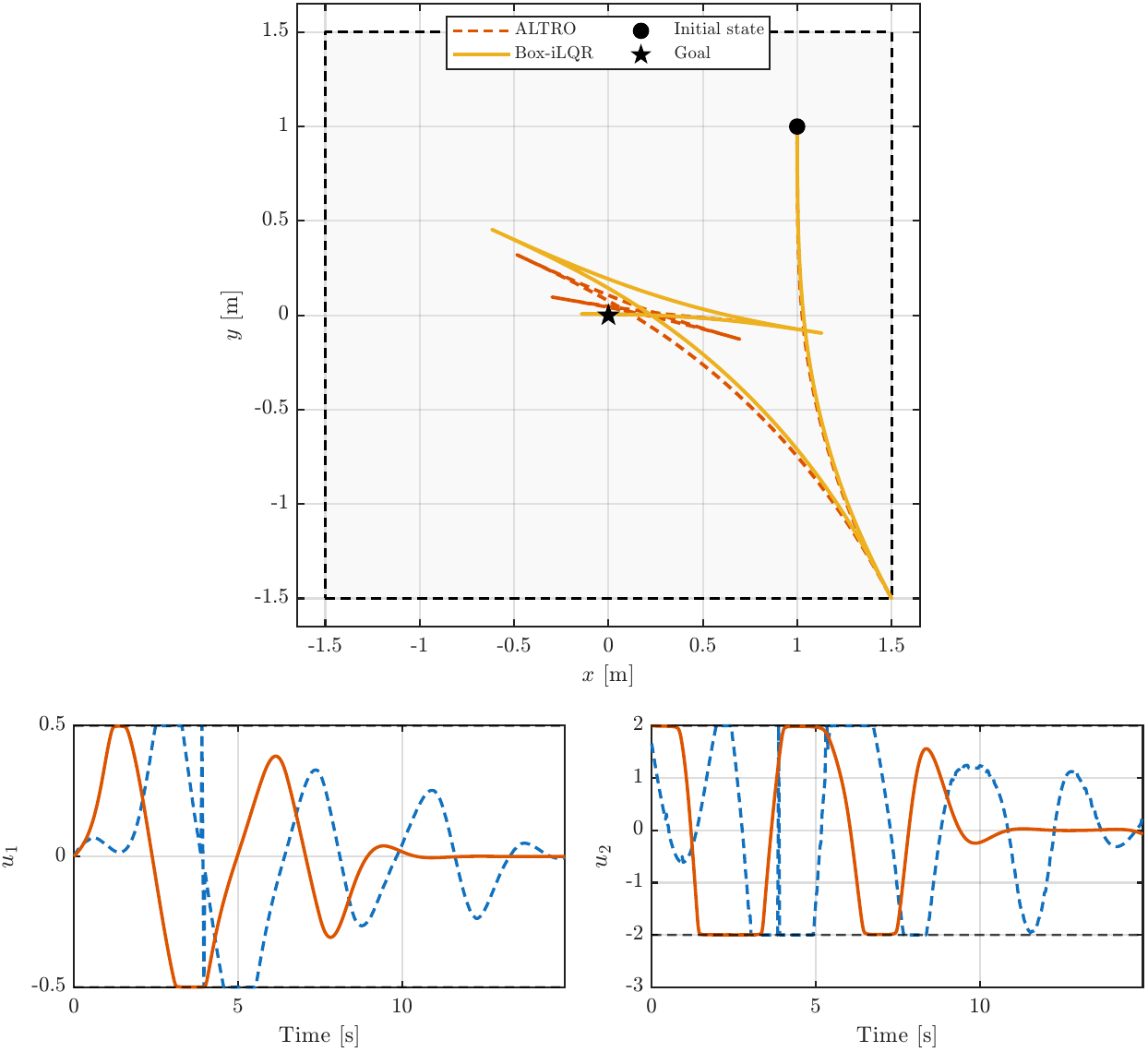}
    \caption{State- and control-constrained car parking. Both trajectories satisfy
the position constraints $-1.5\leq x,y\leq1.5$. Both Box-iLQR and ALTRO reach
the desired terminal configuration $x_g=[0,0,2\pi,0]^\top$, with Box-iLQR attaining a slightly lower nominal cost. The lower panels show the corresponding
control histories and bounds.}
    \label{fig:car_state_constrained}
    \vspace{-0.5cm}
\end{figure}

These results indicate that Box-iLQR is not uniformly superior in iteration
count or nominal objective value. Its main advantage is instead the ability to
treat state and control box constraints in the same Riccati-based formulation
while retaining a local feedback policy.

Box-iLQR maintains strict feasibility of every accepted
iterate, whereas the augmented-Lagrangian mechanism used by ALTRO does not
require intermediate trajectories to remain feasible. This property is particularly useful when intermediate trajectories are
passed to a simulator, hardware system, or another planning layer that assumes
constraint-feasible states and controls.

\subsection{Finite-Barrier Feedback}

We next examine whether the barrier value that best approximates the
hard-constrained optimum also provides the most useful feedback controller.
For stochastic execution, the applied input is
\begin{equation}
u_k^{(q)}
=
\sat_{[\ubar,\uupper]}
\left(
u_{k,\mathrm{fb}}^{(q)}
+
\rho
|u_{k,\mathrm{fb}}^{(q)}|
\odot\xi_k^{(q)}
\right),
\label{eq:noise}
\end{equation}
where
$
u_{k,\mathrm{fb}}^{(q)}
=
u_k^\star
+
K_k(x_k^{(q)}-x_k^\star),
\qquad
\xi_k^{(q)}\sim\mathcal N(0,I).
$
Thus, the disturbance magnitude scales with the current feedback-corrected
input.

For the control-constrained cart-pole, feedback obtained at
$\sigma\simeq0.03$ is compared with Box-iLQR continued to approximately
$\sigma=10^{-6}$ and with CL-DDP. As shown in Fig.~\ref{fig:sigma_feedback_comparison}, the finite-barrier
feedback produces substantially smaller terminal errors ($||x_{N} - x_{g}||_{2}$) over the tested
uncertainty range than the aggressively continued Box-iLQR solution and
CL-DDP. The finite-barrier nominal solution still
reaches the upright equilibrium with terminal error of order $10^{-3}$ in $1000$ Monte Carlo realizations.
 Using the feedback, the terminal error is almost zero for feedback at barrier parameter $\approx0.03$. The shaded region in Fig.~\ref{fig:sigma_feedback_comparison} denotes
$\bar e_N \pm 3s_e$, where $s_e$ is the sample standard deviation of the
terminal error. In addition to achieving a lower mean terminal error, the
finite-barrier Box-iLQR policy exhibits a narrow $3s_e$ envelope, indicating
low variability across the Monte Carlo realizations. Thus, its improved
closed-loop performance is not limited to the mean response; the terminal
errors remain tightly concentrated across the tested realizations.

\begin{figure}[t]
    \centering
    \includegraphics[width=\columnwidth]
    {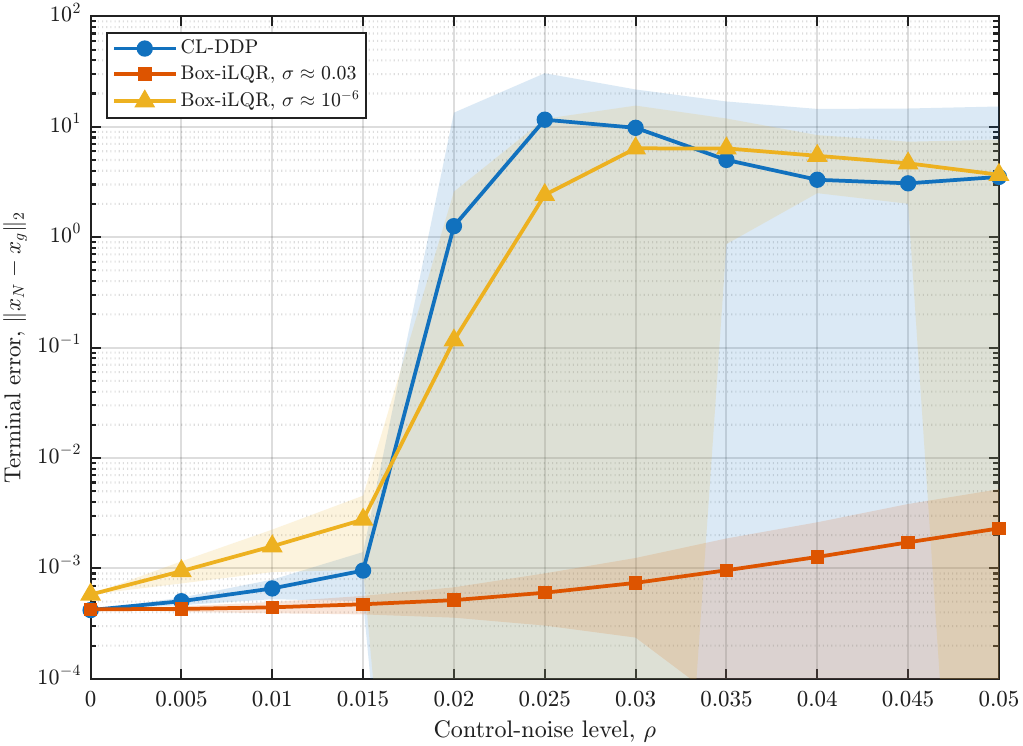}
    \caption{Cart-pole terminal error under execution uncertainty.
    Solid curves show the mean over 1000 Monte Carlo realizations, and
    the shaded regions denote $\bar e_N \pm 3s_e$, where $s_e$ is the
    sample standard deviation of the terminal error. Box-iLQR feedback
    at $\sigma\simeq0.03$ is compared with Box-iLQR at
    $\sigma\simeq10^{-6}$ and CL-DDP.}
    \label{fig:sigma_feedback_comparison}
    \vspace{-0.5cm}
\end{figure}

For the state-constrained cart-pole, we retain the larger barrier
$\gamma_{\mathrm{fb}}\approx0.02$ to provide additional interior margin and perform $100$
closed-loop simulations at $\rho=0.05$. Figure~\ref{fig:state_constraint_mc}
shows all four states and the applied control. No cart-position constraint
violation is observed in the tested realizations. The state trajectories have
very small dispersion and therefore largely overlap the nominal trajectory;
the individual Monte Carlo realizations are more readily visible in the
control history.

\begin{figure}[t]
    \centering
    \includegraphics[width=\columnwidth]
    {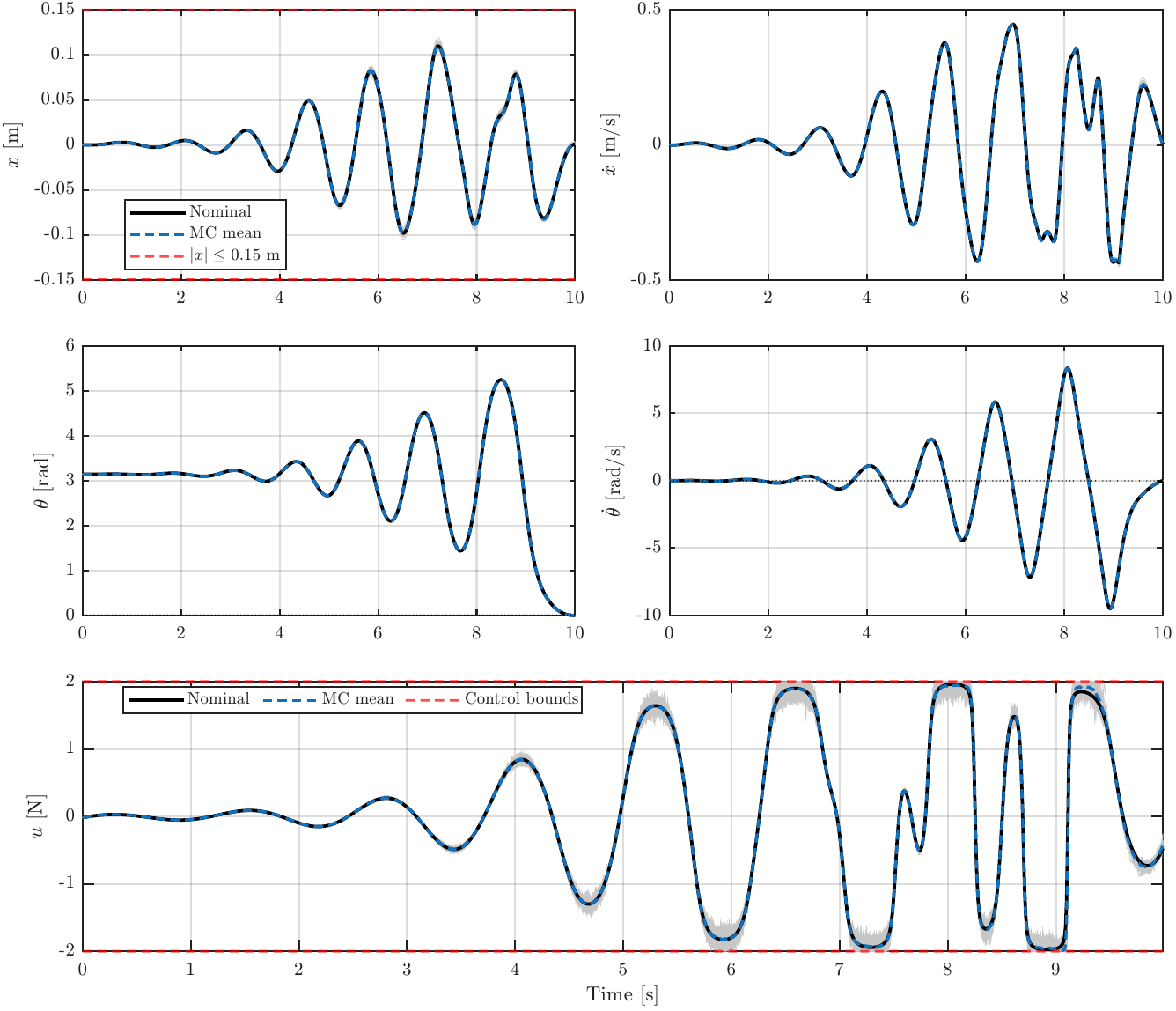}
    \caption{State-constrained cart-pole feedback at $\sigma\simeq \mu \simeq 0.02$ and
    $\rho=0.05$. Gray curves denote $100$ Monte Carlo realizations. The constraints are $|x|\leq0.15$~m and $|u|\leq2$;
    no state violation is observed in the tested set.}
    \label{fig:state_constraint_mc}
\end{figure}

We perform the same closed-loop evaluation for the state-constrained car-parking problem. Figure~\ref{fig:car_feedback} shows $100$ Monte Carlo realizations at a
finite barrier value of approximately $\sigma\approx0.02$ and noise level
$\rho=0.05$.  The controls are saturated at their prescribed bounds,
$|u_1|\leq0.5$ and $|u_2|\leq2$, while the position constraints are
$|x|,|y|\leq1.5$~m.

We compare open-loop execution,
$
    u_k=u_k^\star,
$
with the Box-iLQR feedback policy,
$
    u_k
    =
    u_k^\star
    +
    K_k(x_k-x_k^\star),
$
under identical realizations of the multiplicative execution uncertainty in
\eqref{eq:noise}. Table~\ref{tab:robustness} summarizes the results. For the
car-parking problem, $29$ of $100$ open-loop realizations violate the state
constraints, whereas no violation is observed among the $100$ realizations
using Box-iLQR feedback. Box-iLQR likewise achieves a 100\% success rate for
the state-constrained cart-pole problem. These results demonstrate that the
feedback substantially improves robustness across both state-constrained
examples. A closed-loop comparison with ALTRO is not included because ALTRO is used typically
to compute constrained nominal trajectories and does not directly provide a feedback policy that can ensure constraint satisfaction under uncertainty as analyzed in this section. In addition,
its augmented-Lagrangian iterations need not remain feasible at
intermediate steps, unlike the accepted iterates of Box-iLQR, thereby not allowing to use a feedback associated with a sub-optimal iterate as used by Box-ILQR.

\begin{table}[t]
\vspace{-0.2cm}
  \centering
  \footnotesize
  \setlength{\tabcolsep}{4pt}
  \caption{State-constraint violations over 100 Monte Carlo runs.}
  \label{tab:robustness}
  \begin{tabular}{@{}lcc@{}}
    \toprule
    Problem & Open loop & Box-iLQR (finite-barrier) \\
    \midrule
    Car parking (state) & 29/100 & 0/100 \\
    Cart-pole (state)   & 100/100     & 0/100 \\
    \bottomrule
  \end{tabular}
  \vspace{-0.5cm}
\end{table}

\begin{figure}[t]
    \centering
    \includegraphics[width=\columnwidth]
    {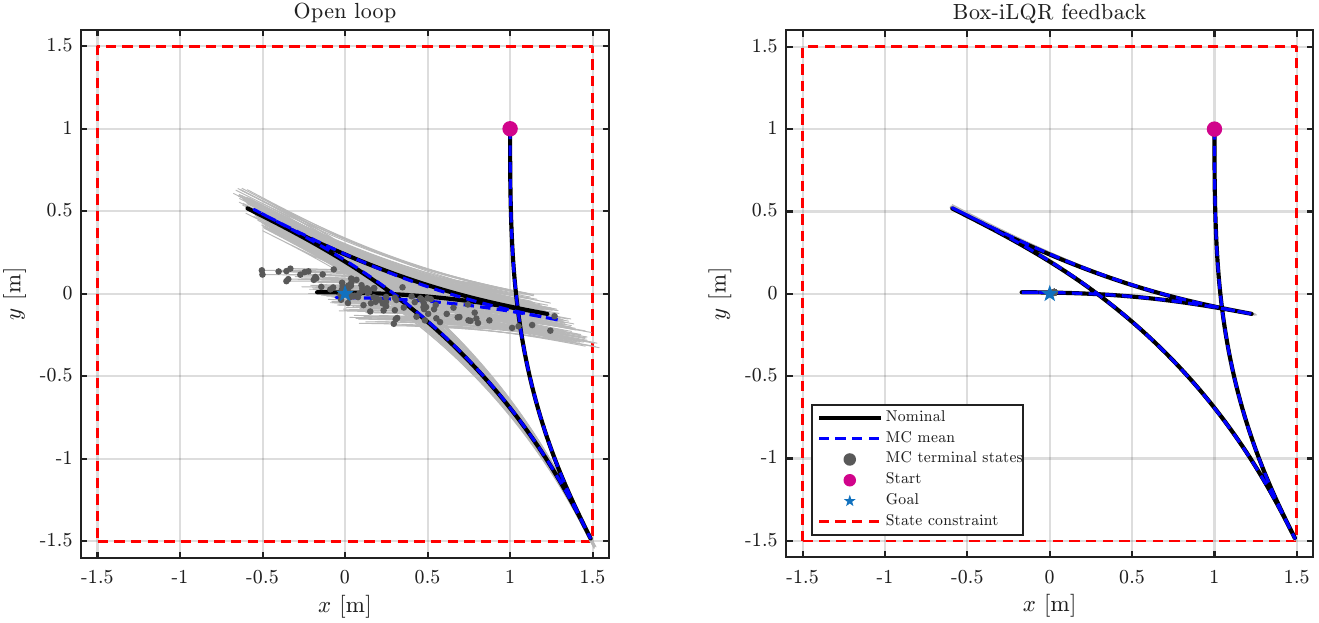}
    \caption{State-constrained car parking under identical execution disturbances
    with $\rho=0.05$. Gray curves show $100$ Monte Carlo realizations, with the
    nominal trajectory in solid black and the Monte Carlo mean dashed. The red
    boundary denotes $|x|,|y|\leq1.5$~m. Open-loop execution violates the state
    constraints in $29/100$ runs, while no violation is observed with
    finite-barrier Box-iLQR feedback at $\gamma_{\mathrm{fb}}\approx0.02$.}
    \label{fig:car_feedback}
\end{figure}

\subsection{Data-Driven MuJoCo Fish}
\label{sec:fish}

The final experiment applies Box-iLQR to the six-input MuJoCo fish
\cite{todorov2012mujoco}, with $|u_j|\leq1$, $j=1,\ldots,6$. Unlike the
analytical benchmarks, the local dynamics are identified from black-box
simulator rollouts following the D2C/search-for-feedback viewpoint
\cite{rand2c,wang2022d2c}. Around the nominal trajectory,
the linear perturbation model
\begin{equation}
    \delta x_{k+1}
    \approx
    \widehat A_k\delta x_k
    +
    \widehat B_k\delta u_k
    \label{eq:data_ltv}
\end{equation}
is estimated by least squares from $200$ symmetric perturbation rollouts.
The identified matrices replace the analytical Jacobians in the Box-iLQR
backward pass. The task is
to drive the fish toward
$
    p_g=
    \begin{bmatrix}
        0 & 0.4 & 0.2
    \end{bmatrix}^{\top}.
$




Both iLQR and Box-iLQR drove the fish to the target. Figure~\ref{fig:fish_control_comparison} compares the resulting Box-iLQR
control histories with unconstrained iLQR. The unconstrained solution requests
inputs outside the admissible interval, whereas data-driven Box-iLQR keeps all
six controls within their prescribed bounds. This demonstrates that the same
barrier formulation can handle multiple simultaneously constrained actuators
even when the local plant Jacobians are obtained entirely from rollout data.

\begin{figure}[t]
    \centering
    \includegraphics[width=\linewidth]
    {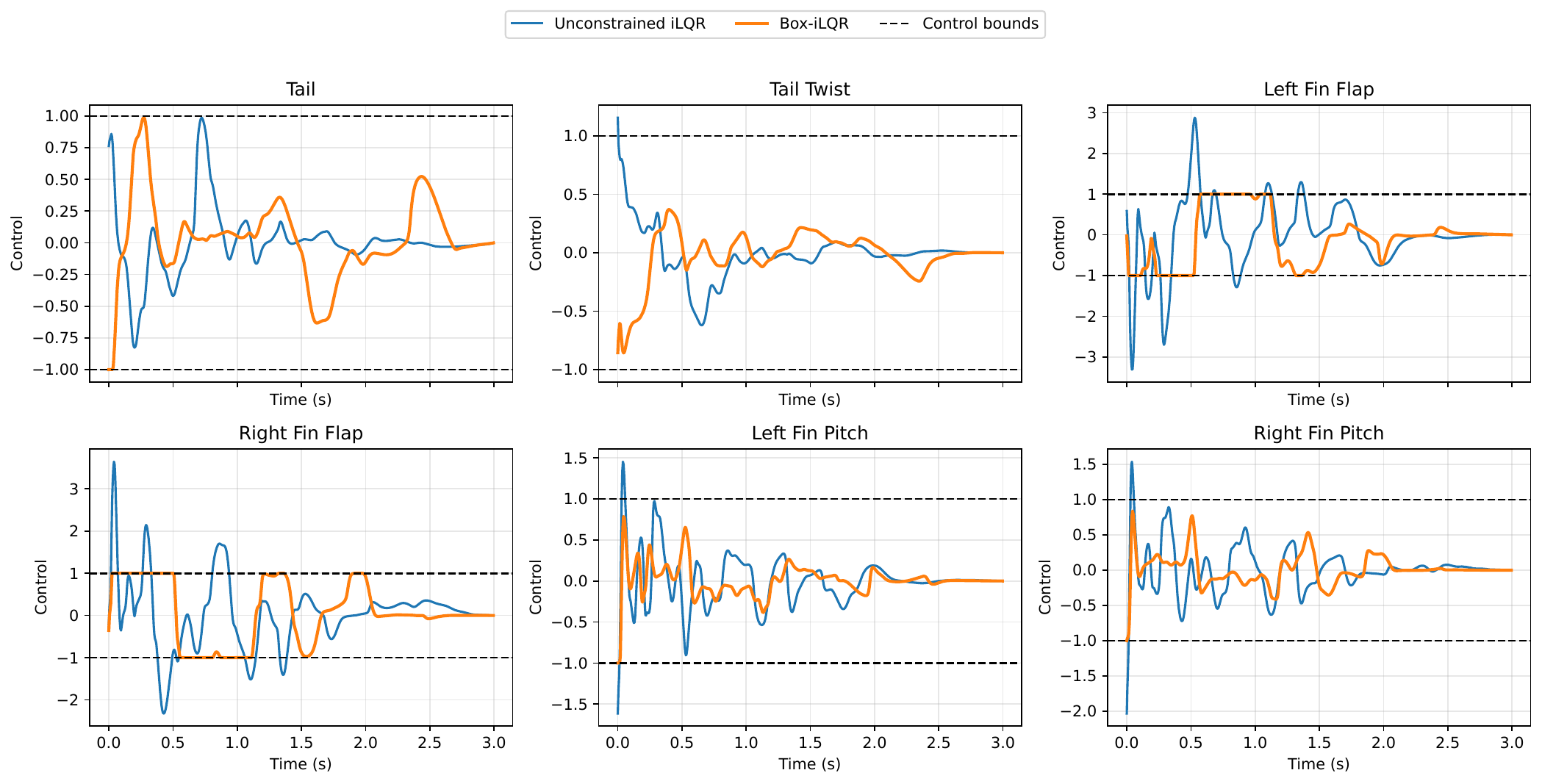}
    \caption{Six-input MuJoCo fish controls. Unconstrained iLQR exceeds the
    actuator limits, whereas data-driven Box-iLQR satisfies all six bounds.}
    \label{fig:fish_control_comparison}
    \vspace{-0.5cm}
\end{figure}

Closed-loop robustness is evaluated for
$\rho\in\{0.01,0.02,\ldots,0.50\}$ using $1000$ realizations per noise level.
To quantify the deviation caused solely by execution uncertainty, we measure
the terminal position error relative to the nominal Box-iLQR terminal
position:
$
    e_N^{(q)}
    =
    \left\|
    p_f^{(q)} - p_f^{\mathrm{nom}}
    \right\|_2 ,
$
where $p_f^{(q)}$ is the terminal fish position in the $q$th Monte Carlo
realization and $p_f^{\mathrm{nom}}$ is the nominal terminal position.
Figure~\ref{fig:fish_terminal_error} reports the mean terminal error and
$3s_{e}$ dispersion, where $s_{e}$ denotes the standard deviation. The results show that finite-barrier feedback recovered from identified local
dynamics remains effective across the tested noise range, demonstrating that
Box-iLQR can combine constrained trajectory optimization and local feedback
recovery without analytical plant Jacobians. Even at $\rho=0.3$, the
$\bar e_N \pm 3s_e$ envelope remains below $0.05$, indicating tightly
concentrated terminal position errors despite $30\%$ multiplicative control
noise.

\begin{figure}[t]
    \centering
    \includegraphics[width=\columnwidth]
    {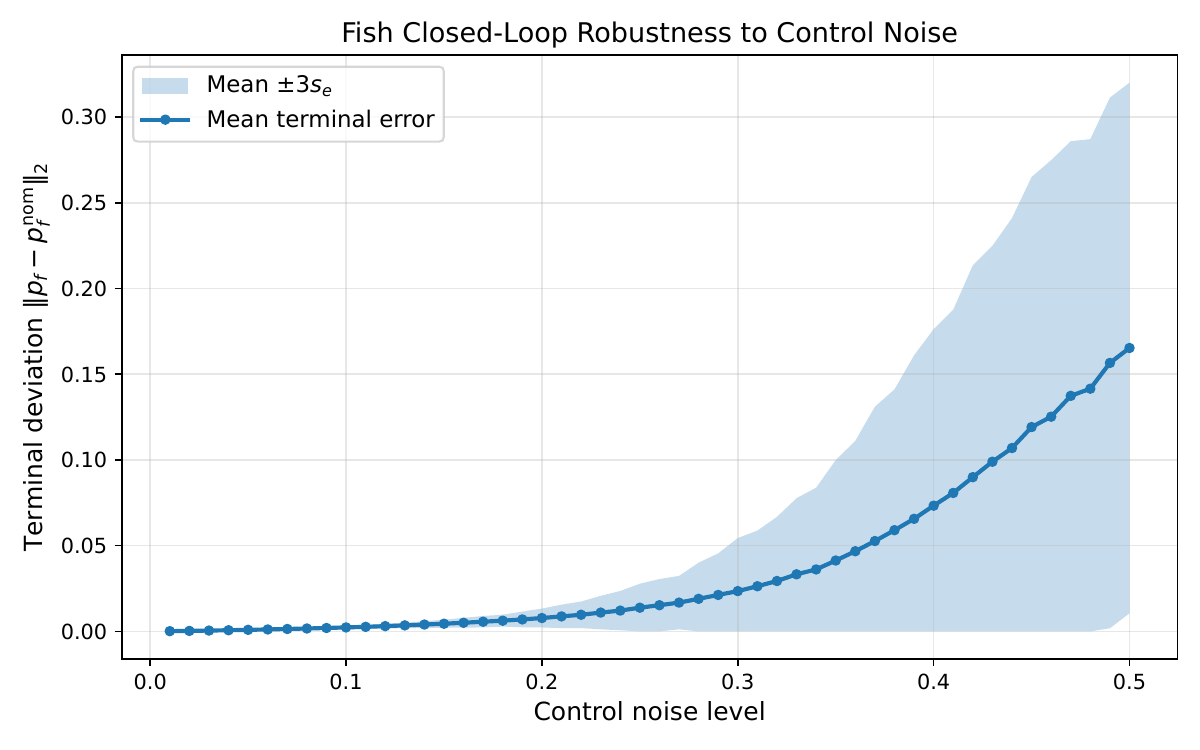}
    \caption{MuJoCo fish terminal position error relative to the nominal
    terminal position. The solid curve shows the mean over 1000 Monte Carlo
    realizations and the shaded region denotes $\bar e_N\pm3s_e$, where $s_e$
    is the sample standard deviation of the terminal position error.}
    \label{fig:fish_terminal_error}
    \vspace{-0.5cm}
\end{figure}

\section{Conclusion}
\label{sec:conclusion}

This paper presented Box-iLQR, a log-barrier iLQR method for box-constrained
states and controls that preserves the standard backward--forward structure
and local affine feedback policy of unconstrained iLQR. iLQR is chosen over full DDP because its local approximation provides more
reliable descent behavior away from the optimum~\cite{abhi_DDP_ILQR_ACC}. 

The barrier admits a useful feedback interpretation. As a constrained
control approaches saturation, its curvature increases and smoothly attenuates
the corresponding feedback row, while state-barrier curvature propagates
backward through the value function and influences feedback before a future
state boundary is encountered. This motivates retaining a finite barrier when
the feedback law, rather than only the nominal constrained optimum, is the
desired output. Unlike penalty-based constrained trajectory optimization, the log-barrier
formulation maintains strict feasibility of accepted iterates, while still
returning the local feedback policy associated with the constrained
trajectory.

Across pendulum, cart-pole, Acrobot, and car-parking benchmarks, Box-iLQR
produces competitive control-constrained solutions and directly handles the
state-constrained cases considered here. For state-constrained car parking, both ALTRO and Box-iLQR reach the desired
terminal configuration, with comparable nominal costs. Closed-loop cart-pole and car-parking experiments
demonstrate the value of the barrier-modified feedback policy under execution
uncertainty. Retaining a finite barrier improves closed-loop behavior in the
cart-pole and car-parking example, while Box-iLQR feedback eliminates the state-constraint
violations observed under open-loop car-parking execution in the tested Monte
Carlo set. The six-input MuJoCo fish experiment further shows that the same
constrained feedback framework can use local dynamics identified entirely
from simulator rollouts.

Future work will focus on automatically selecting finite barrier values using disturbance or covariance information, developing infeasible-start and receding-horizon variants, and validating the method on physical robotic platforms. We also plan to extend the feedback framework beyond box constraints to more general state and control constraints.

\bibliographystyle{IEEEtran}
\bibliography{references_icra2027}

\end{document}